\documentclass[a4paper,11pt]{article}
\usepackage{stmaryrd}
\usepackage{bbm}
\usepackage[top=2.5cm, bottom=2.6cm, left=3cm, right=2.6cm]{geometry}
\usepackage{amsfonts}
\usepackage{mathrsfs}
\usepackage{amsmath}
\usepackage{latexsym}
\usepackage{color}
\usepackage{amssymb}
\usepackage{amsthm}
\usepackage{latexsym}
\usepackage{indentfirst}

\newtheorem{theorem}{Theorem}[section]
\newtheorem{corollary}{Corollary}[section]
\newtheorem{lemma}{Lemma}[section]
\newtheorem{proposition}{Proposition}[section]

\begin{document}

\title{The representations of two-parameter Kashiwara algebras }
\author{Weideng Cui }
\date{}
\maketitle \abstract{In this paper, we give a unified construction of two-parameter quantum groups $U_{r,s}(\mathfrak{g})$, two-parameter Kashiwara algebras $B_{r,s}(\mathfrak{g})$, two-parameter quantized Weyl algebras $W_{r,s}(\mathfrak{g})$ and the action of $U_{r,s}(\mathfrak{g})$ on $B_{r,s}(\mathfrak{g})$ and $W_{r,s}(\mathfrak{g})$. Applying their properties to the category $\mathcal{O}(B_{r,s}(\mathfrak{g}))$, whose objects are ``upper bounded" $B_{r,s}(\mathfrak{g})$-modules, We can give a conceptual proof of its semi-simplicity and the classification of simple modules in it.}\\

 \large

\medskip
\section{Introduction}

The notion of quantum groups was introduced by Drinfeld and Jimbo, independently, around 1985 in their study of the quantum Yang-Baxter equations and solvable lattice models. Quantum groups $U_{q}(\mathfrak{g}),$ depending on a single parameter $q$, are certain families of Hopf algebras that are deformations of universal enveloping algebras of symmetrizable Kac-Moody algebras $\mathfrak{g}$.

In [K], Kashiwara introduced the so-called Kashiwara algebra $B_{q}^{\vee}(\mathfrak{g})$ (he called it the reduced $q$-analogue) and gave the projector $P$ for the Kashiwara algebra of $\mathfrak{s}\mathfrak{l}_{2}$-case in order to define the crystal basis of $U_{q}^{-}(\mathfrak{g}).$ Moreover, let $\mathcal{O}(B_{q}^{\vee}(\mathfrak{g}))$ be the category of $B_{q}^{\vee}(\mathfrak{g})$-modules satisfying a finiteness condition, then he affirmed without proof that $\mathcal{O}(B_{q}^{\vee}(\mathfrak{g}))$ is semi-simple and $U_{q}^{-}(\mathfrak{g})$ is a unique isomorphic classes of simple objects of $\mathcal{O}(B_{q}^{\vee}(\mathfrak{g}))$.

In [N1], Nakashima studied the so-called $q$-Boson Kashiwara algebra $B_{q}(\mathfrak{g})$ and found therein an interesting object $\Gamma$. In [N2], he re-defined the extremal projector $\Gamma$ for $B_{q}(\mathfrak{g})$, clarified its properties and applied it to the representation theory of $q$-Boson Kashiwara algebras. By using the properties of $\Gamma,$ he showed that the category $\mathcal{O}(B_{q}(\mathfrak{g}))$, whose objects are ``upper bounded" $B_{q}(\mathfrak{g})$-modules, is semi-simple and classified its simple modules.

In [F], Fang has given a unified construction of quantum groups, $q$-Boson Kashiwara algebras and quantized Weyl algebras and an action of quantum groups on quantized Weyl algebras. This enables him to give a conceptual proof of the main results in [N2].

From down-up algebras approach, Berkart and Witherspoon recovered Takeuchi's definition of two-parameter quantum groups of type A and investigated their structures and finite dimensional representation theory in [BW1, BW2, BW3, BW4]. Since then, a systematic study of the two-parameter quantum groups of other types has been going on. For instance, Bergeron, Gao and Hu developed the corresponding theory of two-parameter quantum groups for type B, C, D in [BGH1, BGH2]. Later on, Hu et al. continued this project (see [HS, BH] for exceptional types $G_2, E_6, E_7, E_8$). In [HP1, HP2], Hu and Pei give a simpler and unified definition for a class of two-parameter quantum groups $U_{r,s}(\mathfrak{g})$ associated to a finite-dimensional semi-simple Lie algebra $\mathfrak{g}$ in terms of the Euler form.

In [C], we have generalized Nakashima's results to the two-parameter case. Specifically, we define a two-parameter Kashiwara algebra $B_{r,s}(\mathfrak{g})$ and its extremal projector $\Gamma$, and study their basic properties. Applying their properties to the category $\mathcal{O}(B_{r,s}(\mathfrak{g}))$, whose objects are ``upper bounded" $B_{r,s}(\mathfrak{g})$-modules, we can give a proof of its semi-simplicity and the classification of simple modules in it.

In this paper, we will generalize Fang's results to the two-parameter case. Specifically, we give a unified construction of two-parameter quantum groups $U_{r,s}(\mathfrak{g})$, two-parameter Kashiwara algebras $B_{r,s}(\mathfrak{g})$, two-parameter quantized Weyl algebras $W_{r,s}(\mathfrak{g})$ and the action of $U_{r,s}(\mathfrak{g})$ on $B_{r,s}(\mathfrak{g})$ and $W_{r,s}(\mathfrak{g})$, which enables us to give a conceptual proof of the main results in [C].

The organization of this article is as follows. In Sect.~2, we review some notions in Hopf algebras and give an action of quantum doubles on Heisenberg doubles with the help of Schr\"{o}dinger representations following [F]. In Sect.~3, we construct a two-parameter quantum group $U_{r,s}(\mathfrak{g})$, a two-parameter Kashiwara algebra $B_{r,s}(\mathfrak{g})$ and a two-parameter quantized Weyl algebra $W_{r,s}(\mathfrak{g})$ concretely and calculate the action between them in the $\mathfrak{s}\mathfrak{l}_{2}$-case. In Sect.~4, we construct two-parameter quantum Weyl algebras from the braiding in the $U_{r,s}(\mathfrak{g})$-Yetter-Drinfel'd category and prove the main theorem for the structure of $\mathcal{O}(B_{r,s}(\mathfrak{g}))$. At last, we compare our projection with those defined in [C] in the $\mathfrak{s}\mathfrak{l}_{2}$-case.

\section{Skew Hopf pairings and Double constructions}

We must emphasize that the whole section all follows from [F, Section 2]. Suppose that we are working in the complex field $\mathbb{C}$ and all tensor products are over $\mathbb{C}$.

\subsection{Yetter-Drinfel'd modules}

Let $H$ be a Hopf algebra. A vector space $V$ is called a (left) $H$-Yetter-Drinfel'd module if it is simultaneously an $H$-module and an $H$-comodule which satisfy the Yetter-Drinfel'd compatibility condition: for any $h\in H$ and $v\in V,$\[\sum h_{(1)}v_{(-1)}\otimes h_{(2)}\cdot v_{(0)}=\sum (h_{(1)}\cdot v)_{(-1)}h_{(2)}\otimes (h_{(1)}\cdot v)_{(0)},\]
where $\Delta(h)=\sum h_{(1)}\otimes h_{(2)}$ and $\rho(v)=\sum v_{(-1)}\otimes v_{(0)}$ are sweedler notations for coproduct and comodule structure maps.

Morphisms between two $H$-Yetter-Drinfel'd modules are linear maps preserving $H$-module and $H$-comodule structures. We denote the category of $H$-Yetter-Drinfel'd modules by ${}_{H}^{H}\mathcal{YD},$ which is a tensor category. The advantage of Yetter-Drinfel'd module is: for $V, W\in {}_{H}^{H}\mathcal{YD},$ there is a braiding $\sigma : V\otimes W\rightarrow W\otimes V,$ given by $\sigma(v\otimes w)=\sum v_{(-1)}\cdot w\otimes v_{(0)}.$ If both $V$ and $W$ are $H$-module algebras, $V\otimes W$ will have an algebra structure if we use $\sigma$ instead of the usual flip, and we denote it by $V\underline{\otimes}W$.

\subsection{Braided Hopf algebras in ${}_{H}^{H}\mathcal{YD}$}

$\mathbf{Definition ~2.1.}$ (see [AS, Section 1.3]). A braided Hopf algebra in the category ${}_{H}^{H}\mathcal{YD}$ is a collection ($A, m, \eta, \Delta, \varepsilon, S$) such that\vskip2mm
(1) ($A, m, \eta$) is an algebra in ${}_{H}^{H}\mathcal{YD};$ ($A, \Delta, \varepsilon$) is a coalgebra in ${}_{H}^{H}\mathcal{YD}$. That is to say, $m, \eta, \Delta, \varepsilon$ are morphisms in ${}_{H}^{H}\mathcal{YD}.$

(2) $\Delta : A\rightarrow A\underline{\otimes}A$ ia a morphism of algebras.

(3) $\varepsilon : A\rightarrow \mathbb{C}$ and $\eta : \mathbb{C}\rightarrow A$ are algebra morphisms.

(4) $S$ is the convolution inverse of Id$_{A}\in$ End($A$).

\subsection{Braided Hopf modules}

Let $B$ be a braided Hopf algebra in some Yetter-Drinfel'd module category. For a left-braided $B$-Hopf module $M,$ we mean a left $B$-module and a left $B$-comodule satisfying the following compatibility condition:\[\rho\circ l=(m\otimes l)\circ(\mathrm{id}\otimes \sigma \otimes \mathrm{id})\circ(\Delta\otimes \rho): B\otimes M\rightarrow B\otimes M,\]
where $m$ is the multiplication in $B,$ $l :B\otimes M\rightarrow M$ is the module structure map, $\rho : M\rightarrow B\otimes M$ is the comodule structure map and $\sigma$ is the braiding in the fixed Yetter-Drinfel'd module category.

$\mathbf{Example ~2.1.}$ Let $V$ be a vector space over $\mathbb{C}.$ Then $B\otimes V$ admits a trivial $B$-braided Hopf module structure given by: for $b, b'\in B$ and $v\in V,$\[b'\cdot (b\otimes v)=b'b\otimes v,~~~\rho(b\otimes v)=\sum b_{(1)}\otimes b_{(2)}\otimes v\in B\otimes (B\otimes V).\]

We let ${}_{B}^{B}\mathcal{M}$ denote the category of left-braided $B$-Hopf modules. The following proposition gives the triviality of such modules.
\begin{proposition}{\rm (see [F, Proposition~2.1])}
Let $M\in {}_{B}^{B}\mathcal{M}$ be a braided Hopf module, $\rho : M\rightarrow B\otimes M$ be the structural map, $M^{co\rho}=\{m\in M|~\rho(m)=1\otimes m\}$ be the set of coinvariants. Then there exists an isomorphism of braided $B$-Hopf modules:\[M\cong B\otimes M^{co\rho},\]
where the right-hand side adopts the trivial Hopf module structure. Moreover, maps in two directions are given by:\[M\rightarrow B\otimes M^{co\rho},~~~m\mapsto \sum m_{(-1)}\otimes P(m_{(0)}),\]\[B\otimes M^{co\rho}\rightarrow M,~~~b\otimes m\mapsto bm,\]
where $m\in M,$ $b\in B$ and $P : M\rightarrow M^{co\rho}$ is defined by $P(m)=\sum S(m_{(-1)})m_{(0)}.$
\end{proposition}

\subsection{Skew Hopf pairings}

Let $A$ and $B$ be two Hopf algebras with invertible antipodes. A skew Hopf pairing between $A$ and $B$ is a bilinear form $\varphi :A\times B\rightarrow \mathbb{C}$ such that\vskip2mm

(1) for any $a\in A,$ $b, b'\in B,$ $\varphi(a, bb')=\sum \varphi(a_{(1)}, b)\varphi(a_{(2)}, b');$

(2) for any $a, a'\in A,$ $b\in B,$ $\varphi(aa', b)=\sum \varphi(a', b_{(1)})\varphi(a, b_{(2)});$

(3) for any $a\in A,$ $b\in B,$ $\varphi(a,1)=\varepsilon(a),$ $\varphi(1, b)=\varepsilon(b).$

\subsection{Quantum Doubles}

Let $A$ and $B$ be two Hopf algebras with invertible antipodes and $\varphi$ be a skew Hopf pairing between them. The quantum double $D_{\varphi}(A, B)$ is defined as follows:\vskip2mm

(1) As a vector space, it is $A\otimes B.$

(2) As a coalgebra, it is the tensor product of coalgebras $A$ and $B.$

(3) As an algebra, the multiplication is given by:\[(a\otimes b)(a'\otimes b')=\sum \varphi(S^{-1}(a'_{(1)}), b_{(1)})\varphi(a'_{(3)}, b_{(3)})aa'_{(2)}\otimes b_{(2)}b'.\]

\subsection{Schr\"{o}dinger representations}

The Schr\"{o}dinger representation of $D_{\varphi}(A, B)$ on $A$ is given by: for $a, x\in A,$ $b\in B,$\vskip2mm$\hspace{4cm}(a\otimes 1)\cdot x=\sum a_{(1)}x S(a_{(2)}),$\vskip2mm$\hspace{4cm}(1\otimes b)\cdot x=\sum \varphi(x_{(1)}, S(b))x_{(2)}.$

The Schr\"{o}dinger representation of $D_{\varphi}(A, B)$ on $B$ is given by: for $a\in A,$ $b, y\in B,$\vskip2mm$\hspace{4cm}(a\otimes 1)\cdot y=\sum \varphi(a, y_{(1)})y_{(2)},$\vskip2mm$\hspace{4cm}(1\otimes b)\cdot y=\sum b_{(1)}y S(b_{(2)}).$

\begin{proposition}{\rm (see [M, Example~7.1.8])}
With the definition above, both $A$ and $B$ are $D_{\varphi}(A, B)$-module algebras.
\end{proposition}

\subsection{Heisenberg doubles}
The Heisenberg double $H_{\varphi}(A, B)$ is an algebra defined as follows:\vskip2mm

(1) As a vector space, it is $B\otimes A$ and we denote the pure tensor by $b\sharp a.$

(2) The product is given by: for $a, a'\in A,$ $b, b'\in B,$\[(b\sharp a)(b'\sharp a')=\sum \varphi(a_{(1)}, b'_{(1)})bb'_{(2)}\sharp a_{(2)}a'.\]

\subsection{Quantum double action on the Heisenberg double}

We define an action of $D_{\varphi}(A, B)$ on $H_{\varphi}(A, B)$ as follows: for $a, a'\in A,$ $b, b'\in B,$ \[(a\otimes b)\cdot (b'\sharp a')=\sum (a_{(1)}\otimes b_{(1)})\cdot b'\sharp (a_{(2)}\otimes b_{(2)})\cdot a',\]
this is a diagonal type action. Moreover, we have the following result.
\begin{proposition}{\rm (see [F, Proposition~2.3])}
With the action as above, $H_{\varphi}(A, B)$ is a $D_{\varphi}(A, B)$-module algebra.
\end{proposition}

We define a $D_{\varphi}(A, B)$-comodule structure on both $A$ and $B$ as follows:\[A\rightarrow D_{\varphi}(A, B)\otimes A,~~~~a\mapsto \sum a_{(1)}\otimes 1\otimes a_{(2)},\]\[B\rightarrow D_{\varphi}(A, B)\otimes B,~~~~b\mapsto \sum 1\otimes b_{(1)}\otimes b_{(2)}.\]

\begin{proposition}{\rm (see [F, Proposition~2.4])}
With Schr\"{o}dinger representations and comodule structure maps defined as above, both $A$ and $B$ are in the category ${}_{D_{\varphi}}^{D_{\varphi}}\mathcal{YD}.$
\end{proposition}

More generally, we have the following result.

\begin{proposition}{\rm (see [F, Proposition~2.5])}
With the comodule structure map defined by \[\delta : H_{\varphi}(A, B)\rightarrow D_{\varphi}(A, B)\otimes H_{\varphi}(A, B),\]\[b\sharp a\mapsto \sum ((1\otimes b_{(1)})(a_{(1)}\otimes 1))\otimes b_{(2)}\sharp a_{(2)},\]
for $a\in A,$ $b\in B,$ $H_{\varphi}(A, B)$ is in the category ${}_{D_{\varphi}}^{D_{\varphi}}\mathcal{YD}.$
\end{proposition}

\section{Constructions of two-parameter Kashiwara algebras}

This section is devoted to the construction of three important quantum algebras: two-parameter quantum groups, two-parameter Kashiwara algebras and two-parameter quantized Weyl algebras from the machinery built in the last section. Throughout this paper, we always assume that $rs^{-1}\in \mathbb{C}^{*}$ is not a root of unity.

\subsection{Constructions of two-parameter quantum groups $U_{r,s}(\mathfrak{g})$}

Let us start with some notations. For $n>0,$ set\\
$$(n)_{v}=1+v+v^{2}+\cdots+v^{n-1}=\frac{v^{n}-1}{v-1};$$
$$~(n)_{v}^{!}=(1)_{v}(2)_{v}\cdots (n)_{v}~~~~ \mbox{and} ~~~~(0)_{v}^{!}=1;$$
$$\hspace*{-0.6cm}\binom{n}{k}_{v}=\frac{(n)_{v}^{!}}{(k)_{v}^{!}(n-k)_{v}^{!}}~~~~\mbox{for}~n\geq k\geq 0.$$

Let $\mathfrak{g}$ be a finite-dimensional complex semi-simple Lie algebra and $A=(a_{ij})_{i,j\in I}$ be the corresponding Cartan matrix of finite type. Let $\{d_{i}|i\in I\}$ be a set of relatively prime positive integers such that $d_{i}a_{ij}=d_{j}a_{ji}$ for $i,j\in I.$ Let $\Pi=\{\alpha_{i}|i\in I\}$ be the set of simple roots, $Q=\oplus_{i\in I}\mathbb{Z}\alpha_{i}$ a root lattice, $Q^{+}=\oplus_{i\in I}\mathbb{N}\alpha_{i}$ the semigroup generated by positive roots, and $\mathcal{P}$ a weight lattice.

Denote $r_i=r^{d_{i}}, s_{i}=s^{d_{i}}$ for $i\in I.$ As in [HP1], let $\langle\cdot, \cdot\rangle$ be the Euler bilinear form on $Q\times Q$ defined by

$$\langle i, j\rangle :=\langle\alpha_i, \alpha_j \rangle=\begin{cases}d_{i}a_{ij}& \hbox {if } i< j; \\ d_i& \hbox {if }
i=j;\\ 0& \hbox {if }
i>j.
\end{cases}$$

For $\lambda\in \mathcal{P},$ we linearly extend the bilinear form $\langle\cdot, \cdot\rangle$ to be defined on $\mathcal{P} \times \mathcal{P}$ such that $\langle\lambda, i\rangle=\langle\lambda, \alpha_i\rangle=\frac{1}{m}\sum\limits_{j=1}^{n}a_{j}\langle j, i\rangle,$ or $\langle i, \lambda\rangle=\langle \alpha_i, \lambda\rangle=\frac{1}{m}\sum\limits_{j=1}^{n}a_{j}\langle i, j\rangle$ for $\lambda=\frac{1}{m}\sum\limits_{j=1}^{n}a_{j}\alpha_j$ with $a_{j}\in \mathbb{Z}.$

Bricks of our construction are Hopf algebras $U_{r,s}^{\geq }(\mathfrak{g})$ and $U_{r,s}^{\leq }(\mathfrak{g})$, which are defined by generators and relations:\vskip2mm

(1) $U_{r,s}^{\geq }(\mathfrak{g})$ is generated by $E_{i},$ $K_{i}'^{\pm 1}$ ($i\in I$), with the following relations:

\[K_{i}'^{\pm 1}K_{i}'^{\mp 1}=1,~~~~K_{i}'E_{j}K_{i}'^{-1}=r^{-\langle i, j\rangle}s^{\langle j, i\rangle}E_{j};\]
\[\sum\limits_{k=0}^{1-a_{ij}}(-1)^{k}\binom{1-a_{ij}}{k}_{r_{i}s_{i}^{-1}} c_{ij}^{k}E_{i}^{1-a_{ij}-k}E_{j}E_{i}^{k}=0 ~~~(i\neq j),\]

where $c_{ij}^{k}=(r_{i}s_{i}^{-1})^{\frac{k(k-1)}{2}}r^{k\langle j, i\rangle}s^{-k\langle i, j\rangle},$ for $i\neq j.$

It has a Hopf algebra structure given by
\[\Delta(K_{i}'^{\pm 1})=K_{i}'^{\pm 1}\otimes K_{i}'^{\pm 1},~~~
\Delta(E_{i})=E_{i}\otimes K_{i}'+1\otimes E_{i};\]
\[\varepsilon(K_{i}'^{\pm 1})=1,~~~\varepsilon(E_{i})=0,~~~S(E_{i})=-E_{i}K_{i}'^{-1},~~~S(K_{i}'^{\pm 1})=K_{i}'^{\mp 1}.\]

(2) $U_{r,s}^{\leq }(\mathfrak{g})$ is generated by $F_{i},$ $K_{i}^{\pm 1}$ ($i\in I$), with the following relations:

\[K_{i}^{\pm 1}K_{i}^{\mp 1}=1,~~~~K_{i}F_{j}K_{i}^{-1}=r^{-\langle j, i\rangle}s^{\langle i, j\rangle}F_{j};\]
\[\sum\limits_{k=0}^{1-a_{ij}}(-1)^{k}\binom{1-a_{ij}}{k}_{r_{i}s_{i}^{-1}} c_{ij}^{k}F_{i}^{k}F_{j}F_{i}^{1-a_{ij}-k}=0 ~~~(i\neq j).\]

It has a Hopf algebra structure given by
\[\Delta(K_{i}^{\pm 1})=K_{i}^{\pm 1}\otimes K_{i}^{\pm 1},~~~
\Delta(F_{i})=F_{i}\otimes 1+K_{i}\otimes F_{i};\]
\[\varepsilon(K_{i}^{\pm 1})=1,~~~\varepsilon(F_{i})=0,~~~S(F_{i})=-K_{i}^{-1}F_{i},~~~S(K_{i}^{\pm 1})=K_{i}^{\mp 1}.\]

For each $\mu=\sum\limits_{i\in I}\mu_{i}\alpha_{i}\in Q,$ we define elements $K_{\mu}$ and $K_{\mu}'$ by $$K_{\mu}=\prod_{i\in I}K_{i}^{\mu_{i}},~~~~~~K_{\mu}'=\prod_{i\in I}K_{i}'^{\mu_{i}}.$$

Denote by $D_{\varphi}(U_{r,s}^{\geq }(\mathfrak{g}), U_{r,s}^{\leq }(\mathfrak{g}))$ the quantum double of $U_{r,s}^{\geq }(\mathfrak{g})$ and $U_{r,s}^{\leq }(\mathfrak{g})$, where the skew Hopf pairing $\varphi :U_{r,s}^{\geq }(\mathfrak{g})\times U_{r,s}^{\leq }(\mathfrak{g})\rightarrow \mathbb{C}$ is given by $$\varphi(E_{i}, F_{j})=\frac{\delta_{ij}}{s_{i}-r_{i}},~~~\varphi(K_{\lambda}',K_{\mu})=r^{\langle \lambda, \mu\rangle}s^{-\langle\mu, \lambda\rangle},~~~\varphi(E_{i}, K_{j})=\varphi(K_{i}', F_{j})=0.$$

From the definition of the multiplication in quantum double, we have $$(1\otimes F_{j})(E_{i}\otimes 1)=\varphi(S^{-1}(E_{i}), F_{j})K_{i}'\otimes 1+E_{i}\otimes F_{j}+ \varphi(E_{i}, F_{j})\cdot 1\otimes K_{j},$$
that is to say, $$E_{i}F_{j}-F_{j}E_{i}=\delta_{ij}\frac{K_{i}-K_{i}'}{r_{i}-s_{i}}.$$

In the same manner, we also have\[F_{j}K_{i}'=r^{-\langle i, j\rangle}s^{\langle j, i\rangle}K_{i}'F_{j}, ~~~K_{i}E_{j}K_{i}^{-1}=r^{\langle j, i\rangle}s^{-\langle i, j\rangle}E_{j}K_{i}.\]

Then the two-parameter quantum group associated to $\mathfrak{g}$ is just $U_{r,s}(\mathfrak{g})=D_{\varphi}(U_{r,s}^{\geq }(\mathfrak{g}), U_{r,s}^{\leq }(\mathfrak{g}))$, which has been defined in [HP1] and [HP2].

\subsection{Constructions of two-parameter Kashiwara algebras $B_{r,s}(\mathfrak{g})$}

The above procedure, once applied to the Heisenberg double, will give a two-parameter Kashiwara algebra. In this subsection, we will use the different notations for the generators of $U_{r,s}^{\geq }(\mathfrak{g})=\langle e_{i}, \omega_{i}'^{\pm 1}\rangle_{i\in I},$ $U_{r,s}^{\leq }(\mathfrak{g})=\langle f_{i}, \omega_{i}^{\pm 1}\rangle_{i\in I}$ for making it distinct from the quantum double case.

Now, we can compute the multiplication structure between $U_{r,s}^{\geq }(\mathfrak{g})$ and $U_{r,s}^{\leq }(\mathfrak{g})$:
\[(1\sharp \omega_{i}')(f_{j}\sharp 1)=\varphi(\omega_{i}',\omega_{j})=r^{\langle i, j\rangle}s^{-\langle j, i\rangle}f_{j}\sharp \omega_{i}',~~(1\sharp e_{i})(\omega_{j}\sharp 1)=\omega_{j}\sharp e_{i}.\]

For this reason, it is better to adopt generators $e_{i}'=(s_{i}-r_{i})\omega_{i}'^{-1}e_{i}$, then we have \[\Delta(e_{i}')=e_{i}'\otimes 1+ \omega_{i}'^{-1}\otimes e_{i}',~~\omega_{\lambda}'e_{i}'\omega_{\lambda}'^{-1}=r^{-\langle \lambda, i\rangle}s^{\langle i, \lambda\rangle}e_{i}',~~\varphi(e_{i}', f_{j})=\delta_{ij}.\]

Then, in this case we have \[(1\sharp e_{i}')(\omega_{j}\sharp 1)=\varphi(\omega_{i}'^{-1}, \omega_{j})\omega_{j}\sharp e_{i}'=r^{-\langle i, j\rangle}s^{\langle j, i\rangle}\omega_{j}\sharp e_{i}',\]
which is what we have desired, that is, $\omega_{j}e_{i}'\omega_{j}^{-1}=r^{\langle i, j\rangle}s^{-\langle j, i\rangle}e_{i}'.$

We calculate the relation between $e_{i}'$ and $f_{j}$ as follows:\[(1\sharp e_{i}')(f_{j}\sharp 1)=\varphi(e_{i}', f_{j})\cdot 1+ \varphi(\omega_{i}'^{-1}, \omega_{j})f_{j}\sharp e_{i}'=r^{-\langle i, j\rangle}s^{\langle j, i\rangle}f_{j}\sharp e_{i}',\]
which is equivalent to the following form: \[e_{i}'f_{j}=r^{-\langle i, j\rangle}s^{\langle j, i\rangle}f_{j} e_{i}'+\delta_{ij}.\]

Thus all the relations have been recovered and we will set $B_{r,s}(\mathfrak{g}) :=H_{\varphi}(U_{r,s}^{\geq }(\mathfrak{g}), U_{r,s}^{\leq }(\mathfrak{g})),$ where $B_{r,s}(\mathfrak{g})$ is just the two-parameter Kashiwara algebra defined in [C], up to the following isomorphism:\[e_{i}'\mapsto e_{i}'',~~f_{i}\mapsto f_{i},~~\omega_{i}^{\pm 1}\mapsto \omega_{i}'^{\pm 1},~~\omega_{i}'^{\pm 1}\mapsto \omega_{i}^{\pm 1},~~r\mapsto s,~~s\mapsto r.\]

We will use the abbreviated notations $U_{r,s}, B_{r,s}, U_{r,s}^{\geq }, U_{r,s}^{\leq }$ for $U_{r,s}(\mathfrak{g}), B_{r,s}(\mathfrak{g}),$ $U_{r,s}^{\geq }(\mathfrak{g}), U_{r,s}^{\leq }(\mathfrak{g})$ if there is no confusion.

\subsection{Action of quantum doubles on Heisenberg doubles}

From Proposition 2.5, we get an action of $D_{\varphi}(U_{r,s}^{\geq }, U_{r,s}^{\leq })$ on $H_{\varphi}(U_{r,s}^{\geq }, U_{r,s}^{\leq })$ such that $H_{\varphi}(U_{r,s}^{\geq }, U_{r,s}^{\leq })$ is a $D_{\varphi}(U_{r,s}^{\geq }, U_{r,s}^{\leq })$-Yetter-Drinfel'd module, that is, $B_{r,s}$ is a $U_{r,s}$-Yetter-Drinfel'd module. Denote by $W_{r, s}=W_{r, s}(\mathfrak{g})$ the subalgebra of $B_{r,s}$ generated by $e_{i}'$ and $f_{i}$ ($i\in I$), which is called a two-parameter quantized Weyl algebra.

From the definition of Schr\"{o}dinger representation and Proposition 2.5, we have
\begin{proposition}

$B_{r,s}$ and $W_{r, s}$ are both $U_{r,s}$-module algebra; moreover, they are $U_{r,s}$-Yetter-Drinfel'd modules.
\end{proposition}

\subsection{Examples}

In this subsection, we will compute the action of quantum double on Heisenberg double in the $\mathfrak{s}\mathfrak{l}_{2}$-case. Generators of $D_{\varphi}(U_{r,s}^{\geq }, U_{r,s}^{\leq })$ are $E, F, K^{\pm 1},$ \\$K'^{\pm 1}$; for $H_{\varphi}(U_{r,s}^{\geq }, U_{r,s}^{\leq })$, they are $e', f, \omega^{\pm 1}, \omega'^{\pm 1}.$

At first, we calculate the action of $K, K'$:\[K'\cdot e'=\omega'e'\omega'^{-1}=r^{-1}s~ e',~~~~K\cdot e'=\varphi(\omega'^{-1}, \omega^{-1})e'=rs^{-1} e',\]
\[K'\cdot f=\varphi(\omega', \omega)f=rs^{-1} f,~~~~K\cdot f=\omega f \omega^{-1}=r^{-1}s f.\]

Then, we will also give all the other actions, suppose that $m\leq n:$\vskip3mm $E^{m}\cdot e'^{n}=\frac{(n+m-1)_{rs^{-1}}^{!}}{(n-1)_{rs^{-1}}^{!}}(rs^{-1})^{-nm-\frac{m(m-1)}{2}}\cdot r^{-m} e'^{n+m};$\vskip3mm
$E^{m}\cdot f^{n}=\frac{1}{(s-r)^{m}}\cdot\frac{(n)_{rs^{-1}}^{!}}{(n-m)_{rs^{-1}}^{!}} f^{n-m};$\vskip3mm
$F^{m}\cdot e'^{n}=(-1)^{m}\cdot\frac{(n)_{rs^{-1}}^{!}}{(n-m)_{rs^{-1}}^{!}}\cdot (rs^{-1})^{m} e'^{n-m};$ \vskip3mm $F^{m}\cdot f^{n}=\prod_{i=0}^{m-1} (1-(rs^{-1})^{-(n+i)}) f^{n+m}.$ \vskip2mm

These formulas will be useful in the next section.

\section{Modules over two-parameter Kashiwara algebras}
Define $B_{r, s}^{+}$ and $B_{r, s}^{-}$ as the subalgebras of $B_{r,s}$ that are generated by $e_{i}',$ $f_{j}$ ($i,j\in I$) respectively, and $B_{r,s}^{0}$ the subalgebra generated by $\omega_{i}^{\pm 1}, \omega_{i}'^{\pm 1}$ ($i\in I$). Let $U_{r,s}^{0}$ be the subalgebra of $U_{r,s}$ generated by $K_{i}^{\pm 1}, K_{i}'^{\pm 1}$ ($i\in I$).

\subsection{Constructions of $W_{r, s}(\mathfrak{g})$ from braiding}

We have seen in the previous section that $W_{r, s}(\mathfrak{g})$ is in ${}_{U_{r,s}}^{U_{r,s}}\mathcal{YD}.$

On $B_{r, s}^{+}$, there is a $U_{r,s}$-Yetter-Drinfel'd module algebra structure: the $U_{r,s}$-module structure is given by the Schr\"{o}dinger representation and the $U_{r,s}$-comodule structure is given by $\delta(e_{i}')=(s_{i}-r_{i})\cdot K_{i}'^{-1}E_{i}\otimes 1+ K_{i}'^{-1} \otimes e_{i}'.$ These structures are compatible because $\delta$ is just $\Delta$ in $U_{r,s}.$

On $B_{r, s}^{-}$, there is a $U_{r,s}$-Yetter-Drinfel'd module algebra structure: the $U_{r,s}$-module structure is given by the Schr\"{o}dinger representation and the $U_{r,s}$-comodule structure is given by $\delta(f_{i})= F_{i}\otimes 1+ K_{i} \otimes f_{i}.$ These structures are compatible because $\delta$ is just $\Delta$ in $U_{r,s}.$

In the category ${}_{U_{r,s}}^{U_{r,s}}\mathcal{YD}$, we can use the braiding to give the tensor product of two module algebras a structure of algebra. For our purpose, consider $W_{r, s}\otimes W_{r, s},$ and denote the braiding by $\sigma;$ then $(m\otimes m)\circ(\mathrm{id}\otimes \sigma \otimes \mathrm{id}):$\[W_{r, s}\otimes W_{r, s}\otimes W_{r, s}\otimes W_{r, s}\rightarrow W_{r, s}\otimes W_{r, s}\otimes W_{r, s}\otimes W_{r, s}\rightarrow W_{r, s}\otimes W_{r, s},\] gives $W_{r, s}\otimes W_{r, s}$ a structure of algebra. We will denote this algebra by $W_{r, s}\underline{\otimes} W_{r, s}.$

We want to restrict this braiding to the subspace $B_{r, s}^{-}\otimes B_{r, s}^{+}\subset W_{r, s}\otimes W_{r, s}$. This requires to restrict the $U_{r,s}$-comodule structure on $W_{r, s}$ to $B_{r, s}^{+}$ and the $U_{r,s}$-module structure on $W_{r, s}$ to $B_{r, s}^{-}.$ This comodule structure could be directly restricted as we did in the beginning of this section; the possibility for the restriction of the module structure comes from the fact that the Schr\"{o}dinger representation makes $B_{r, s}^{-}$ stable. This gives an algebra $B_{r, s}^{-}\underline{\otimes}B_{r, s}^{+}.$

We calculate the product in $B_{r, s}^{-}\underline{\otimes}B_{r, s}^{+}:$\vskip2mm
$~~~~~~(f_{i}\otimes 1)(1\otimes e_{j}')=f_{i}\otimes e_{j}',$
\begin{align*}(1\otimes e_{i}')(f_{j}\otimes 1)&=\sum (e_{i}')_{(-1)}\cdot f_{j} \otimes (e_{i}')_{(0)}\\
&=((s_{i}-r_{i})\cdot K_{i}'^{-1}E_{i})\cdot f_{j}\otimes 1+ K_{i}'^{-1}\cdot f_{j}\otimes e_{i}'\\
&=\delta_{ij}+ r^{-\langle i, j\rangle}s^{\langle j,i\rangle}f_{j}\otimes e_{i}'.
\end{align*}

This is nothing but the relations in $W_{r, s}(\mathfrak{g})$. Moreover, as a vector space, $W_{r, s}(\mathfrak{g})$ has a decomposition $W_{r, s}(\mathfrak{g})\cong B_{r, s}^{-} \otimes B_{r, s}^{+},$ and the inverse map is given by the multiplication; so we have the following proposition.
\begin{proposition}

There exists an algebra isomorphism:\[B_{r, s}^{-}\underline{\otimes}B_{r, s}^{+}\cong W_{r, s}(\mathfrak{g}),~~~~f\otimes e\mapsto fe,\]
where $f\in B_{r, s}^{-},$ $e\in B_{r, s}^{+}.$

\end{proposition}

\subsection{Modules over $W_{r, s}(\mathfrak{g})$}

This subsection is devoted to studying modules over $W_{r, s}(\mathfrak{g})$ with finiteness conditions.

We define the category $\mathcal{O}(W_{r, s})$ as a full subcategory of $W_{r, s}(\mathfrak{g})$-modules which contains those $W_{r, s}(\mathfrak{g})$-modules satisfying: for any $M$ in $\mathcal{O}(W_{r, s})$ and any $m\in M,$ there exists an integer $l> 0$ such that for any $i_{1}, i_{2},\ldots, i_{l}\in I,$ we have $e_{i_{1}}'e_{i_{2}}'\cdots e_{i_{l}}'m=0.$

Let $M$ be a $W_{r, s}(\mathfrak{g})$-module in $\mathcal{O}(W_{r, s})$. The braided Hopf algebras $B_{r, s}^{+}$ and $B_{r, s}^{-}$ are both $\mathbb{N}$-graded by defining deg$(e_{i}')= $deg$(f_{i})=1.$ We denote $B_{r, s}^{+}(n)$ the finite-dimensional subspace of $B_{r, s}^{+}$ containing elements of degree $n$ and $(B_{r, s}^{+})^{g}=\bigoplus_{n\geq 0} B_{r, s}^{+}(n)^{*}$ the graded dual coalgebra of $B_{r, s}^{+}.$

Recall that there exists a pairing between $B_{r, s}^{+}$ and $B_{r, s}^{-}$ given by $\varphi(e_{i}', f_{j})=\delta_{ij}.$ Because $\varphi(B_{r, s}^{+}(n), B_{r, s}^{-}(m))=0$ for $m\neq n$ and the restriction of $\varphi$ on $B_{r, s}^{+}(n)\times B_{r, s}^{-}(n)$ for $n\geq 0$ is non-degenerate, the graded dual of $B_{r, s}^{+}$ is anti-isomorphic to $B_{r, s}^{-}$ as graded coalgebras. Thus we obtain $(B_{r, s}^{+})^{g}\cong B_{r, s}^{-}.$

From the definition, the action of $B_{r, s}^{+}$ on $M$ is locally nilpotent, and so from duality, we obtain a left $(B_{r, s}^{+})^{g}$-comodule structure on $M$. With the help of the isomorphism above, there is a left $B_{r, s}^{-}$-comodule structure on $M$ given by: if we denote $\rho : M\rightarrow B_{r, s}^{-}\otimes M$ by $\rho(m)= \sum m_{(-1)}\otimes m_{(0)},$ then for $e\in B_{r, s}^{+},$\[e\cdot m=\sum \varphi(e, m_{(-1)}) m_{(0)}.\]

Thus from a $W_{r, s}(\mathfrak{g})$-module, we obtain a $B_{r, s}^{-}$-module which is at the same time a $B_{r, s}^{-}$-comodule.

\hspace*{-0.5cm}$\mathbf{Remark ~4.1.}$ Here, for the left $B_{r, s}^{-}$-comodule structure on $M$, it is needed to consider the braided Hopf algebra structure on $B_{r, s}^{-}$, which is considered as an object of ${}_{U_{r,s}^{0}}^{U_{r,s}^{0}}\mathcal{YD};$ that is to say, $\Delta_{0} :B_{r, s}^{-}\rightarrow B_{r, s}^{-}\underline{\otimes}B_{r, s}^{-}, \Delta_{0}(f_{i})=f_{i}\otimes 1+ 1\otimes f_{i}.$ For the left $B_{r, s}^{-}$-module structure on $M,$ we need to consider $B_{r, s}^{-}$ as an object of ${}_{U_{r,s}}^{U_{r,s}}\mathcal{YD}$, so we keep the ordinary coproduct, that is, $\Delta(f_{i})=f_{i}\otimes 1+\omega_{i}\otimes f_{i}\in B_{r, s}^{\leq}\otimes B_{r, s}^{-}.$

We define a linear projection $\pi : B_{r, s}^{\leq}\rightarrow B_{r, s}^{-}$ by $ft\mapsto f\varepsilon(t),$ where $f\in B_{r, s}^{-}$ and $t\in B_{r, s}^{0}.$
\begin{proposition}
For any $f\in B_{r, s}^{-},$ $m\in M,$ we have \[\rho(f\cdot m)= \sum \pi(f_{(1)}m_{(-1)}) \otimes f_{(2)}\cdot m_{(0)}=\Delta_{0}(f)\rho(m).\]
\end{proposition}
\begin{proof}
At first, we calculate $\rho(f\cdot m),$ for any $e\in B_{r, s}^{+},$ we have
\begin{align*}
e(f\cdot m)=(ef)\cdot m&=\sum (e_{(-1)}\cdot f)e_{(0)}m\\
&= \sum \varphi(e_{(0)}, m_{(-1)})(e_{(-1)}\cdot f)m_{(0)}\\
&=\sum \varphi(e_{(-1)}, f_{(1)})\varphi(e_{(0)}, m_{(-1)}) f_{(2)} \cdot m_{(0)}\\
&= \sum \varphi(e, f_{(1)}m_{(-1)}) f_{(2)} \cdot m_{(0)}.
\end{align*}

Here, $f_{(1)}m_{(-1)}$ is not necessarily in $B_{r, s}^{-},$ but we always have \[\varphi(e, f_{(1)}m_{(-1)})=\varphi(e, \pi(f_{(1)}m_{(-1)})),\]
which gives the first equality.

For the second one, from the definition of braiding $\sigma$ in Section 2, we have\[\Delta_{0}(f)\rho(m)= \sum f^{(1)}((f^{(2)})_{(-1)}\cdot m_{(-1)})\otimes (f^{(2)})_{(0)} m_{(0)},\]
where $\Delta_{0}(f)=\sum f^{(1)}\otimes f^{(2)}.$ As said in Remark 4.1, we consider $B_{r, s}^{-}$ as a braided Hopf algebra when considering the comudule structure, so $f^{(2)}=(f^{(2)})_{(0)}=f_{(2)}$ and from the definition of $\Delta$ in $B_{r, s}^{-}$, we can see that \[\pi(f_{(1)})((f_{(2)})_{(-1)}\cdot m_{(-1)})=\pi(f_{(1)}m_{(-1)}).\]
Notice that $\Delta_{0}(f)=(\pi\otimes \mathrm{id})(\Delta(f)).$ So we have \[\sum f^{(1)}\otimes f^{(2)}=\sum \pi(f_{(1)})\otimes f_{(2)}\]
and the above formula gives the second equality.
\end{proof}

Proposition 4.1 shows that the compatibility relation between the module and comodule structures, so from a $W_{r, s}(\mathfrak{g})$-module $M$, we can obtain a braided $B_{r, s}^{-}$-Hopf module. Thus we have the following corollary.
\begin{corollary}
There exists an equivalence of categories $\mathcal{O}(W_{r, s}) \sim {}_{B_{r,s}^{-}}^{B_{r,s}^{-}}\mathcal{M}.$
\end{corollary}

The following theorem gives the structural result for $W_{r, s}(\mathfrak{g})$-modules with finiteness condition.
\begin{theorem}
There exists an equivalence of categories $\mathcal{O}(W_{r, s}) \sim \mathrm{Vect},$ where $\mathrm{Vect}$ is the category of vector spaces. The equivalence is given by \[M\mapsto M^{co\rho},~~~~ V\mapsto B_{r, s}^{-}\otimes V,\]
where $M\in \mathcal{O}(W_{r, s}),$ $V\in \mathrm{Vect},$ and $M^{co\rho}=\{m\in M~|~ \rho(m)=1\otimes m\}$ is the set of coinvariants.
\end{theorem}
\begin{proof}
We have seen in the Corollary 4.1 that $\mathcal{O}(W_{r, s})$ and ${}_{B_{r,s}^{-}}^{B_{r,s}^{-}}\mathcal{M}$ are equivalent; so the theorem comes from the triviality of braided Hopf modules as shown in Proposition 2.1.
\end{proof}

It is better to write down an explicit formula for $\rho$. For $\beta\in Q^{+}\cup \{0\},$ we define \[(B_{r, s}^{+})_{\beta}=\{x\in B_{r, s}^{+}|~\omega_{i}x\omega_{i}^{-1}=r^{\langle\beta, i\rangle}s^{-\langle i, \beta\rangle}x, \omega_{i}'x\omega_{i}'^{-1}=r^{-\langle i, \beta\rangle}s^{\langle \beta, i\rangle}x, \forall~ i\in I\}.\]

We can define $(B_{r, s}^{-})_{-\beta}$ in a similar manner. Elements in $(B_{r, s}^{+})_{\beta}$ (respectively $(B_{r, s}^{-})_{-\beta}$) are called of degree $\beta$ (respectively $-\beta$). Let $e_{\alpha, i}\in (B_{r, s}^{+})_{\alpha},$ $1\leq i\leq \mathrm{dim} ((B_{r, s}^{+})_{\alpha})$ be a basis of $B_{r, s}^{+},$ $f_{\beta, j}$ be the dual basis with respect to $\varphi,$ such that \[\varphi(e_{\alpha, i}, f_{\beta, j})=\delta_{ij}\delta_{\alpha\beta}.\]

We define $\mathcal{R}=\sum_{i,\alpha}f_{\alpha, i}\otimes e_{\alpha, i},$ for $M\in \mathcal{O}(W_{r, s}),$ then $\mathcal{R}(1\otimes m)$ is well-defined for every $m\in M.$

\begin{proposition}
For each $m\in M,$ we have $\rho(m)=\mathcal{R}(1\otimes m).$
\end{proposition}
\begin{proof}
For each $m\in M,$ we can write $\rho(m)=\sum_{j,\alpha}f_{\alpha, j}\otimes m_{\alpha, j},$ then from the definition of left comodule structure,\[e_{\beta, i}\cdot m=\sum\limits_{j,\alpha} \varphi(e_{\beta, i}, f_{\alpha, j})m_{\alpha, j}=m_{\beta, i}.\]
So $\rho(m)=\sum_{i,\alpha}f_{\alpha, i}\otimes e_{\alpha, i}\cdot m=\mathcal{R}(1\otimes m).$
\end{proof}

Now, let us look at an example.

$\mathbf{Example ~4.1.}$ Consider the $\mathfrak{s}\mathfrak{l}_{2}$-case, generators of $B_{r, s}(\mathfrak{s}\mathfrak{l}_{2})$ will be denoted by $e, f, \omega^{\pm 1}, \omega'^{\pm 1}.$ We choose $m\in M$ such that $e\cdot m\neq 0,$ $e^{2}\cdot m=0.$ From $e^{2}f=(r^{-1}s)^{2}fe^{2}+(r^{-1}s+1)e,$ we can get
\begin{align*}
\rho(f\cdot m)&=1\otimes fm+ f\otimes efm+ f^{2}\otimes \frac{e^{2}fm}{r^{-1}s+1}\\&=1\otimes fm+f\otimes r^{-1}s\cdot fem+f\otimes m+f^{2}\otimes em,
\end{align*}
\begin{align*}
(\pi\otimes \mathrm{id})(\Delta(f)\rho(m))&=(\pi\otimes \mathrm{id})(f\otimes m+ f^{2}\otimes em+ \omega\otimes fm+ \omega f\otimes fem)\\
&=f\otimes m+ f^{2}\otimes em+ 1\otimes fm + f\otimes r^{-1}s\cdot fem,
\end{align*}
\begin{align*}
\Delta_{0}(f)\rho(m)=f\otimes m+ f^{2}\otimes em+ 1\otimes fm +r^{-1}s\cdot f\otimes fem.
\end{align*}

For a $W_{r, s}(\mathfrak{g})$-module $M,$ $0\neq m\in M$ is called a maximal vector if it is annihilated by all $e_{i}'$ ($i\in I$). The set of all maximal vectors in $M$ is denoted by $K(M).$ The following lemma is a direct consequence of the definition.
\begin{lemma}
Suppose that $m\in M^{co\rho}.$ Then for any non-constant element $e\in B_{r, s}^{+},$ we have $e\cdot m=0.$
\end{lemma}

\begin{lemma}
Let $f\in B_{r, s}^{-},$ $f\notin \mathbb{C}^{*},$ such that for any $i\in I$, $e_{i}'\cdot f=0.$ Then $f=0.$
\end{lemma}
\begin{proof}
If $e_{i}'\cdot f=0$ for any $i,$ $f$ is annihilated by all non-constant elements in $B_{r, s}^{+}.$ For any $e\in B_{r, s}^{+},$ $e\cdot f=\sum \varphi(e, f_{(1)})f_{(2)},$ we can suppose that these $f_{(2)}$ are linearly independent, so $\varphi(e, f_{(1)})$ for any $f_{(1)}$ and any non-constant $e\in B_{r, s}^{+},$ the non-degeneracy of the Hopf pairing forces $f_{(1)}$ to be constants.

So now $f=(\mathrm{id}\otimes \varepsilon)\Delta(f)=\sum f_{(1)}\varepsilon(f_{(2)})\in \mathbb{C}$ and it must be 0 from the hypothesis.\end{proof}

Combined with Theorem 4.1, Lemma 4.1 and 4.2, we can get

\begin{corollary}
Let $M\in \mathcal{O}(W_{r, s})$ be a $W_{r, s}$-module. Then $M^{co\rho}=K(M).$
\end{corollary}

\subsection{Modules over $B_{r, s}(\mathfrak{g})$}

We define the category $\mathcal{O}(B_{r,s})=\mathcal{O}(B_{r,s}(\mathfrak{g}))$ a full subcategory of left module category over $B_{r, s}(\mathfrak{g})$ containing objects satisfying the following conditions:\vskip3mm

(1) Any object $M$ has a weight space decomposition $M=\oplus_{\lambda\in \mathcal{P}}M_{\lambda},$ where $M_{\lambda}=\{u\in M|~\omega_{\mu}u=r^{\langle\lambda, \mu\rangle}s^{-\langle\mu, \lambda\rangle}u,~\omega_{\mu}'u=r^{-\langle\mu, \lambda\rangle}s^{\langle\lambda, \mu\rangle}u,~\forall \mu\in Q\}$ and $\dim M_{\lambda}<\infty$ for any $\lambda\in \mathcal{P}$.

(2) For any object $u\in M$ there exists $l>0$ such that $e_{i_{1}}'e_{i_{2}}'\cdots e_{i_{l}}'u=0$ for any $i_{1}, i_{2},\ldots, i_{l}\in I.$\vskip3mm

Moreover, we denote $\mathcal{O}'(B_{r,s})$ as the category of $B_{r,s}$-modules satisfying only (2) above.

The main theorem of this paper is the following structural result.

\begin{theorem}
There exists an equivalence of categories $\mathcal{O}'(B_{r,s}) \sim {}_{U_{r,s}^{0}}\mathcal{\mathrm{Mod}}.$ The equivalence is given by \[M\mapsto K(M),~~~~V\mapsto B_{r, s}^{-}\otimes V,\]
where $M\in \mathcal{O}'(B_{r,s}),$ $V$ is a $U_{r,s}^{0}$-module and $K(M)$ is the set of maximal vectors in $M,$ when looked upon as a $W_{r, s}$-module.

Moreover, when restricted to the subcategory $\mathcal{O}(B_{r,s}),$ the equivalence above gives $\mathcal{O}(B_{r,s})\sim {}_{\mathcal{P}}Gr,$ where the latter is the category of $\mathcal{P}$-graded vector spaces.
\end{theorem}
\begin{proof}
For any $U_{r,s}^{0}$-module $V,$ we may regard it as a vector space through the forgetful functor. So from Theorem 4.1, $N=B_{r, s}^{-}\otimes V$ admits a locally finite $W_{r, s}$-module structure such that $N^{co\rho}=V.$ Moreover, if the $U_{r,s}^{0}$-module structure on $V$ is under consideration, there exists a $B_{r,s}$-module structure over $B_{r, s}^{-}\otimes V$ given by: for $v\in V,$ $x, f\in B_{r, s}^{-},$ $e\in B_{r, s}^{+},$ $t=\omega_{i}^{\pm 1}, \omega_{i}'^{\pm 1}~(i\in I)\in B_{r,s}^{0},$ we have$$e\cdot (x\otimes v)=\sum \varphi(e, x_{(1)})x_{(2)}\otimes v,~~f\cdot (x\otimes v)=fx\otimes v,~~t\cdot (x\otimes v)=txt^{-1}\otimes tv.$$
As a summary, we have given a functor ${}_{U_{r,s}^{0}}\mathcal{\mathrm{Mod}}\rightarrow \mathcal{O}'(B_{r,s}).$

From now on, let $M\in \mathcal{O}'(B_{r,s})$ be a $B_{r,s}$-module with finiteness condition.

The restriction from $B_{r,s}$-modules to $W_{r,s}$-modules gives a functor $\mathcal{O}'(B_{r,s})$\\$\rightarrow \mathcal{O}(W_{r,s}),$ thus we obtain a functor $\mathcal{O}'(B_{r,s}) \sim {}_{U_{r,s}^{0}}\mathcal{\mathrm{Mod}}$ by composing with the equivalence functor $\mathcal{O}(W_{r,s})\rightarrow \mathrm{Vect}.$

From Theorem 4.1 and the module structure defined above, these two functors give an equivalence of categories $\mathcal{O}'(B_{r,s}) \sim {}_{U_{r,s}^{0}}\mathcal{\mathrm{Mod}}$. So the first point of this theorem follows from Corollary 4.2.

The second point follows from the equivalence of the $U_{r,s}^{0}$-modules satisfying condition (1) in $\mathcal{O}(B_{r,s})$ and $\mathcal{P}$-graded vector spaces.
\end{proof}

\subsection{Semi-simplicity of $\mathcal{O}(B_{r,s})$}

The following result is a direct corollary of Lemma 4.1 and Theorem 4.2.
\begin{corollary}
Let $M\in \mathcal{O}(B_{r,s})$ be a nontrivial $\mathcal{O}(B_{r,s})$-module. There exists nonzero maximal vectors in $M$.
\end{corollary}

For each $\lambda\in \Lambda,$ we define the $B$-module $H(\lambda)$ by $H(\lambda) :=B/I_{\lambda},$ where the left ideal $I_{\lambda}$ is defined as
$$I_{\lambda}=\sum\limits_{i\in I}B_{r,s} e_{i}'+\sum\limits_{\mu\in Q}B_{r,s} (\omega_{\mu}-r^{\langle\lambda, \mu\rangle}s^{-\langle\mu, \lambda\rangle})+\sum\limits_{\mu'\in Q}B_{r,s} (\omega_{\mu'}'-r^{-\langle\mu', \lambda\rangle}s^{\langle\lambda, \mu'\rangle}),$$
then $H(\lambda)$ is a free $B_{r, s}^{-}$-module of rank one, generated by 1. The following structural results follow from Theorem 4.2.
\begin{corollary}
Let $M\in \mathcal{O}(B_{r,s}),$ $v\in M$ be a maximal vector of weight $\lambda.$ Then $B_{r, s}^{-}\otimes \mathbb{C}v\rightarrow H(\lambda),$ $F\otimes v\mapsto F\cdot v$ is an isomorphism of $B_{r,s}$-modules. In particular, $H(\lambda)$ are all simple $B_{r,s}$-modules.
\end{corollary}
\begin{corollary}
$(1)$ Let $M$ be a simple $B_{r,s}$-module. Then there exists some $\lambda$ such that $M\cong H(\lambda).$

$(2)$ Suppose that $M\in \mathcal{O}(B_{r,s}).$ Then $M$ is semi-simple.
\end{corollary}

\subsection{Extremal projectors}

We will give the projection $P$ in the $\mathfrak{s}\mathfrak{l}_{2}$-case, and show that it is almost the operator given in [C] up to a change $r\mapsto s,$ $s\mapsto r.$

At first, we calculate $\Delta_{0} :B_{r, s}^{-}\rightarrow B_{r, s}^{-}\underline{\otimes}B_{r, s}^{-}$ and the antipode $S$; note that the multiplication is twisted in the right-hand side.
\begin{lemma}
$(1)$ $\Delta_{0}(f^{n})=\sum_{p=0}^{n} \binom{n}{p}_{rs^{-1}} (rs^{-1})^{p(p-n)}f^{p}\otimes f^{n-p};$

$(2)$ $S(f^{n})=(-1)^{n} (rs^{-1})^{\frac{-n(n-1)}{2}}f^{n}.$
\end{lemma}
\begin{proof}
(1) By induction on $n$ and use the following identity:\[\binom{m+1}{n}_{rs^{-1}}=\binom{m}{n}_{rs^{-1}}+(rs^{-1})^{m+1-n}\binom{m}{n-1}_{rs^{-1}}.\]

(2) Applying $S\otimes \mathrm{id}$ to the formula of $\Delta_{0}(f^{n}),$ we can get the desired result by using induction and the following identity:\[\sum\limits_{k=0}^{m}(-1)^{k}\binom{m}{k}_{rs^{-1}}(rs^{-1})^{\frac{k(k-1)}{2}}=0.\]\end{proof}

From Proposition 4.3, in the $\mathfrak{s}\mathfrak{l}_{2}$-case, we have \[\rho(m)=\sum\limits_{n=0}^{\infty} (rs^{-1})^{\frac{n(n-1)}{2}}\frac{f^{n}}{(n)_{rs^{-1}}^{!}}\otimes e^{n}\cdot m,\]
and then ~~~~~~~~~~$P(m)=\sum\limits_{n=0}^{\infty} (-1)^{n}(rs^{-1})^{-\frac{n(n-1)}{2}}\frac{f^{n}}{(n)_{r^{-1}s}^{!}}e^{n}\cdot m.$\vskip2mm
It is almost the extremal operator $\Gamma$ given in [C, Example 6.1] up to a change $r\mapsto s,$ $s\mapsto r.$

\vskip3mm Acknowledgements. The author is deeply indebted to Dr. Xin Fang
for his patient reply to the author's questions.



Mathematical Sciences Center, Tsinghua University, Jin Chun Yuan West Building.

Beijing, 100084, P.R. China.

E-mail address: cwdeng@amss.ac.cn

\end{document}